\documentclass[12pt,twoside,reqno]{amsart}

\usepackage{graphicx}
\usepackage{yhmath}
\usepackage{amssymb}
\usepackage{amsmath}
\usepackage{amsthm}
\usepackage{amsfonts}
\usepackage{amstext}
\usepackage{cite}
\usepackage{stix}

\usepackage[mathscr]{eucal}
\usepackage{xspace}
\usepackage{color}

\newcommand{\R}{\mathbb{R}}

\newcommand{\EE}{\mathcal {E}}

\newcommand{\DD}{\mathcal{D}}

\newcommand{\xx}{\mathbf{x}}

\newcommand{\kk}{\mathbf{k}}

\newcommand{\yy}{\mathbf{y}}
\newcommand{\nn}{\mathbf{n}}

\newcommand{\0}{\mathbf{0}}

\newcommand{\ga}{\alpha}
\newcommand{\gb}{\mathbf{\beta}}
\newcommand{\gc}{\gamma}
\newcommand{\gd}{\delta}

\newcommand{\GD}{\Delta}
\newcommand{\ve}{\varepsilon}
\newcommand{\GO}{\Omega}

\newcommand{\vf}{\varphi}
\newcommand{\gl}{\lambda}
\newcommand{\GC}{\Gamma}
\newcommand{\gtt}{\boldsymbol{\tau}}

\newcommand{\GL}{\Lambda}
\newcommand{\gxx}{\boldsymbol{\xi}}

\newcommand{\dx}{\mathrm{d}\,}

\newcommand{\po}{\partial}

\renewcommand{\div}{\text{\rm div}\,}

\newcommand{\dist}{\text{\rm dist}\,}

\newcommand{\bea}{\begin{eqnarray}}
\newcommand{\eea}{\end{eqnarray}}

\theoremstyle{plain}
\newtheorem{theorem}{Theorem}[section]

\newtheorem{lemma}{Lemma}[section]

\newtheorem{proposition}{Proposition}[section]

\theoremstyle{definition}
\newtheorem{definition}{Definition}[section]
\theoremstyle{remark}

\newtheorem{remark}{Remark}[section]
\numberwithin{equation}{section}

\begin{document}
\title[Elliptic Equations with Discontinuous Coefficients]{Elliptic Equations in  Divergence Form with Discontinuous Coefficients in Domains with corners }

\author{ Jun Chen }
\address{Jun Chen: Three Gorges Mathematical Research Center, China Three Gorges University, Yichang 443002}
\email{\texttt{chenjun4@ctgu.edu.cn}}

\author{Xuemei Deng }
\address{Xuemei Deng:   Three Gorges Mathematical Research Center, China Three Gorges University, Yichang 443002}
\email{\texttt{dxuemei@ctgu.edu.cn}}

\date{}

\begin{abstract}
We study the equations in  divergence form with  piecewise $C^{\ga}$  coefficients. The domains contain corners and the discontinuity surfaces are attached to edges of the corners. We obtain piecewise $C^{1,\ga}$  estimates across the discontinuity surfaces and  provide an example to illustrate the issue about the regularity at the corners.
\end{abstract}

\maketitle

\section{Introduction}\label{sec-intro}

Consider the following elliptic problem 
\begin{align}\label{ellipticu}
	\po_i (a^{ij}  \po_j u ) &= h + \po_i g^i  &&   \text{in} \quad \GO,\\
u &= \vf  &&  \text{on} \quad  \po \GO,\label{ellipticupo}
\end{align}
 where $\GO$ is a bounded domain in $\R^n$ with a closed $(n-1)$-surface as its boundary, 
$a^{ij} \in L^\infty (\GO)$ and the following uniform ellipticity condition is satisfied
\begin{align}\label{unifelliptic}
	&\gl |\gxx|^2 \le  a^{ij} (x) \xi_i\xi_j \le \Lambda  |\gxx|^2, && \forall \xx \in \GO,\quad \gxx = (\xi_1, \cdots, \xi_n) \in \R^n,
\end{align}
where $\gl, \GL$ are two positive constants.

Without the assumption on the continuity of coefficients $a^{ij}$, we can only obtain    H\"older continuity of the solution by De Giogi-Nash estimates ({\it cf.}\cite{gt}, Theorem 8.24). Here arises the question:  if the  coefficients are piecewise  H\"older continuous, can we obtain   piecewise  H\"older estimates for the gradient  of the solution?  In \cite{Liyy}, Li and Vogelius studied such problems arising from the models about materials of fiber-reinforced composite. They showed that under the assumption that the coefficients are piecewise $C^{\ga}$, the solution is  piecewise $C^{1,\bar{\ga}}$ for some $\bar{\ga} \in \left(1, \frac{\ga}{(\ga+1)n}\right)$ and their estimates are independent of the distance between the discontinuity surfaces. Hence, they can deal with the case for two touching discontinuity surfaces.

\begin{figure}[htbp]
	\centering
\begin{minipage}[t]{0.55\textwidth}
	\centering
		\includegraphics[width=6.5cm]{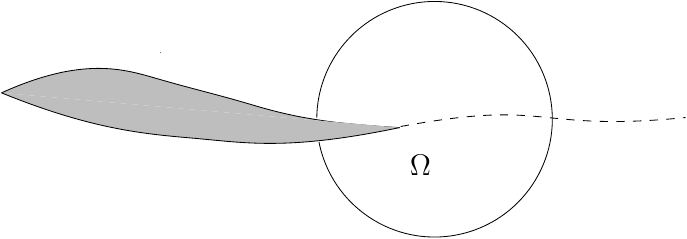}
		\caption{A vortex line attached to an airfoil.}
		\label{Figure-1}%
			\end{minipage}	
	\begin{minipage}[t]{0.43\textwidth}
		\centering
		\includegraphics[width=5cm]{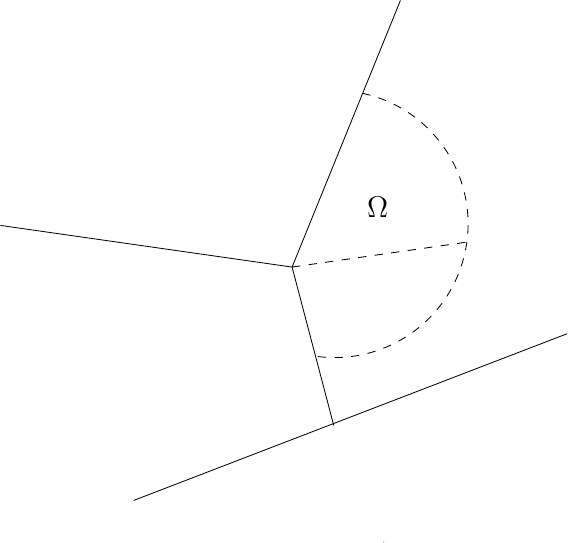}
	\caption{A Mach configuration.}
	\label{Figure-2}
	\end{minipage}%
\end{figure}

In this paper, we are interested in such kind of elliptic problems from conservation laws in aerodynamics. For example, when contact discontinuities happen in  a subsonic flow described by steady compressible Euler equations,   elliptic equations in divergence form with discontinuous coefficients across the contact discontinuity surfaces or lines can be derived from the Euler system ({\it cf.} \cite{Bae,cxz,schen1,schen2}). Very often, there is only one contact discontinuity surface in the domain, which makes the structure of the subdomains simpler than that in  \cite{Liyy}. On the other hand, the equations from the conservation laws are nonlinear in general. One need to linearize the equations and design iteration schemes to solve the nonlinear problems. Hence, in each iteration step when solving the linearized equation, loss of regularity as in \cite{Liyy} is not allowed   in order to close the iteration arguments. That is, if the coefficients are piecewise $C^\ga$ and the discontinuity surface is $C^{1,\ga}$, piecewise $C^{1,\ga}$ estimates are needed rather than piecewise $C^{1,\bar{\ga}}$ estimates with $\bar{\ga} < \ga$. Furthermore, when we study Mach reflection ({\it cf.} \cite{schen1,schen2})  or airfoils with vortex lines ({\it cf.} \cite{cxz}), the contact discontinuity lines are attached to the corners of the domain boundaries (see Figure \ref{Figure-1} and    Figure \ref{Figure-2}). This is a different situation from \cite{BonVo} and \cite{Liyy}, in which  discontinuity surfaces neither stretch to the boundary, nor cross  with corners.  In \cite{BonVo},  Bonnetier and   Vogelius gave an example of discontinuity surfaces crossing with corners. In their example, the solution is not even $W^{1,\infty}$, thus illustrating so called  ``corner effect''. 
 
 This paper tries to understand how the way of intersection between the boundary and discontinuity surface affects the regularity of solutions near   the boundary corners. We will study bounded domains with corners described as follows.

\begin{figure}[htbp]
	\centering
		\includegraphics[width=5cm]{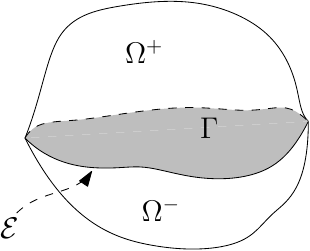}
		\caption{Domain $\Omega$.}
		\label{Figure-3}
\end{figure}
Suppose that domain $\GO$ is divided by $(n-1)$-surface $\Gamma$ into two disjoint open sets $\GO^+$ and $\GO^-$ (see Figure \ref{Figure-3}), i.e.,
\begin{align} \label{con-domain}
\begin{array}{ll}
\GO^+ \cap   \GO^- =\emptyset, &\GO^+ \cup \Gamma\cup \GO^- =\GO,\\ \GC \cap \po \GO=\emptyset, &\overline{\GO^+ }\cap \overline{\GO^- } = \po \GO^+ \cap \po  \GO^-  =\overline{\GC}.
\end{array}
\end{align}
Denote   the upper and lower parts of boundary $\po \GO$ by
\begin{align*} 
(\po \GO)^+:= \overline{\GO^+ } \backslash \GC,&&(\po \GO)^-:= \overline{\GO^- } \backslash \GC
\end{align*}
and the intersection of  $(\po \GO)^+$ and $(\po \GO)^-$ by  
\begin{align*} 
\EE:= (\po \GO)^+ \cap (\po \GO)^- =\po \GC, 
\end{align*}  
which is called the edge of  domain $  \GO$ and is assumed to be a closed $(n-2)$-surface. 

Let $(r,\theta,\xx'):=(r,\theta, x_3, \cdots, x_n)$ be cylindrical coordinates. Assume that $ \theta_-< 0 <\theta_+,   \theta_+ - \theta_- < 2 \pi$.  We call 
\begin{align}\label{def-wedge}
	 W=\{\xx \in \R^n:  r>0,   \theta_-< \theta <\theta_+\}
\end{align}
a wedge and set
\begin{align}\label{domain-w}
&W^+=\{\xx \in \R^n:  0 <\theta <\theta_+ \}, 
	&&W^-=\{\xx \in \R^n:   \theta_- <\theta <0 \}.
\end{align}

\begin{definition}[Wedge condition] \label{def-wedgecon}
	We say that the edge $\EE$ satisfies the wedge condition, if for any $\xx \in \EE$, there exist $r>0$, a wedge $W$, a neighborhood $U$ of $\0$ and a $C^{1,\ga}$ homeomorphism $\chi: B_r(\xx) \to U$, such that 
	\begin{align*} 
		\chi(\GO  \cap B_r(\xx)) = W  \cap U, &&
		\chi(\GO^I \cap B_r(\xx)) = W^I \cap U , &&I=+,-.
	\end{align*}  
\end{definition}

Since loss of regularity of solutions near the edge happens very often, it is convenient to  introduce the following weighted H\"older norms: Suppose $\DD$ is an open domain in $\R^n$ with given boundary portion $E\subset\partial\DD$. 
For any $\xx, \yy $ in  $\DD$,  define 
\begin{align*}
&\gd_\xx :=  \min
(\textrm{dist}(\xx,E),1),& \gd_{\xx,\yy} := \min (\gd_\xx,\gd_{\yy}).
\end{align*}
Let $\ga \in (0,1)$,  $\tau \in
\R$  and $k$ be a nonnegative integer. Let $\kk = (k_1,
 \cdots, k_n)$ be an integer-valued vector, where $k_i \ge 0, i= 1, \cdots, n$,
$|\kk|=k_1  + \cdots + k_n$ and let $D^{\kk}= \partial_{x_1}^{k_1}\cdots \partial_{x_n}^{k_n}$. We
define
\begin{align}
[ f ]_{k,0;\DD}^{(\tau;E)}
&= \sup_{ \begin{subarray}{c}
	\xx\in \DD\\
	|\kk|=k
	\end{subarray}}\left( (\gd_{\xx})^{\max (k+\tau,0)}   |D^\kk f(\xx)|\right),\label{def-normk0}\\
{[ f ]}_{k,\ga;\DD}^{(\tau;E)}&= 
\sup_{
	\begin{subarray}{c}
	\xx, \yy\in \DD\\
	\xx \ne \yy\\
	|\kk|=k
	\end{subarray}}
\left((\gd_{\xx,\yy})^{\max(k+\ga+
	\tau,0)} \frac{|D^\kk f(\xx)-D^\kk
	f(\yy)|}{|\xx-\yy|^\ga}\right),
\label{def-normkga}
\\
\|f\|_{k,\ga; \DD}^{(\tau;E)}&= \sum_{i=0}^k {[ f ]_{i,0;\DD}^{(\tau;E)}} + {[ f ]}_{k,\ga;\DD}^{(\tau;E)}. \label{def-norm}
\end{align}
Denote 
\begin{align*}
C^{k,\ga}_{(\tau;E)}(\DD) := \{ f: \|f\|_{k,\ga; \DD}^{(\tau;E)} <\infty\}.
\end{align*}

\begin{theorem} \label{thmmain}
Assume that \eqref{con-domain} holds, $\po  \GO^+, \po \GO^-, \GC$   are $C^{1,\ga}$ surfaces of $n-1$ dimensions and edge $\EE$ is a $C^{1,\ga}$ surface  of $n-2$ dimensions,   satisfying the wedge condition \eqref{def-wedge}. Suppose that $a^{ij}$ satisfies the uniform ellipticity condition \eqref{unifelliptic}, $a^{ij}, g^i \in C^{\ga}(\overline{\GO^+}) \cap C^{\ga}(\overline{\GO^-}), h \in L^{\infty}(\GO) $ and  $\vf \in C^{1,\ga}(\overline{\GO^+}) \cap C^{1,\ga}(\overline{\GO^-})  \cap C (\overline {\GO }) $. 
Then there exist positive constants $\gb$ and $C$, depending on $n,\ga, \gl,\Lambda, \GO$, such that  there exists a unique solution $u  \in C^{1,\ga}_{(-\gb;\EE)}(\GO^+) \cap C^{1,\ga}_{(-\gb;\EE)}(\GO^-)  \cap C^{\gb} (\overline {\GO }) $ to elliptic problem \eqref{ellipticu} \eqref{ellipticupo} with the following estimate:
\begin{align}\label{est-glo-u}
\max_{I=+,-}\|u\|^{(-\gb;\EE)}_{1,\ga; \GO^I} & \le C (	\max_{I=+,-}\|\varphi\|_{1,\ga;   \GO ^I}  
+  \|h\|_{L^{\infty}( \GO) } + \max_{i=1,\cdots, n; I=+,-}\|g^i\|_{0,\ga; \GO^I }  ).
\end{align} 
\end{theorem}

\begin{remark}
	When $n=3$, the discontinuity surface $\GC$ is  two dimensional  with a closed curve  as its edge  $\EE$; while in the case when $n=2$, $\GC$ is a closed curve with two ends points as the edge $\EE$.
\end{remark}

The organization of this paper is as follows. In Section 2, we establish the Schauder interior estimates across the discontinuity surface through Lemma \ref{lemma1}. In Section 3, we obtain corner estimates in Lemma \ref{lemma2}, which, combined with the interior estimates in Section 2, give  rise to the global estimates. In Section 4, we provide an example to show that different boundary shapes and data can render solutions with different regularity, such as $C^\gc$ or $C^{1,\ga}$  smoothness at the corners.

\section{Interior estimates}

To obtain the interior estimates, we need the following proposition, which is a simplified version of Proposition 3.2  in \cite{Liyy}. In order to state the proposition, we first introduce the following weighted $L^p$ norm in a given domain $\DD$ for any $s>0$ and $p\in (1,\infty)$:
	\begin{align*}
&\|f\|_{Y^{s,p}(\DD)} : = \sup_{0<r\le 1} r^{1-s} \left(\intbar_{r\DD}  | f |^p \right)^{1/p}.
\end{align*}

Let $B_r(\xx)$ be the open ball centered at $\xx$ with radius $r$, $B_r$ be the open ball centered at the origin $\0$ with radius $r$ and 
	\begin{align*}
D_r &:= B_r\cap \{x_n =0\},& B_r^+ &:= B_r\cap \{x_n >0\},& B_r^- &:= B_r\cap \{x_n <0\}.
\end{align*}

\begin{proposition} \label{prop1}
	Suppose that $a^{ij}$ and  $\bar{a}^{ij}$ satisfy the uniform ellipticity condition \eqref{unifelliptic},  $\bar{a}^{ij},\overline{G}^i, \overline{H}$ are constants in both $B_4^+ $ and $ B_4^-$.  Suppose  $ g^i \in L^q(B_4), h \in L^{q/2}(B_4) $ for some $q>n$. Let $\bar{\ga} \in (0,1)$ and $u \in H^1(B_4)$ be a solution to  
	\begin{align}\label{ellipticu1}
		\po_i (a^{ij}  \po_j u ) &= h + \po_i g^i 
	\end{align}
	in $B_4$	with
		\begin{align}
	&\|u\|_{L^\infty (B_4)} \le 1. \label{propcon1}
	\end{align}
		Then there exist   constants $\ve_0 >0$ and  $C >0$, depending on  $n,q, \bar{\ga},, \gl, \Lambda$, such that if the following hold
		\begin{align}
 \max_{i,j=1, \cdots,n}	\|a^{ij} - \bar{a}^{ij}\|_{Y^{1+\bar{\ga},q}(B_4)} & \le \ve_0,\label{propcon2}\\
\max_{i=1, \cdots,n}\|g^i - \overline{G}^i\|_{Y^{1+\bar{\ga},q}(B_4)} + \|h - \overline{H}\|_{Y^{\bar{\ga},q/2}(B_4)}  & \le \ve_0,\label{propcon3}\\
\max_{i=1, \cdots,n}\|\overline{G}^i\|_{L^\infty (B_4)} + \|\overline{H}\|_{L^\infty (B_4)}&\le 1, \label{propcon4}
	\end{align}	
then there exists a funtion $p$,  continuous in $B_1$ and piecewise linear  in  $B_1^+ \cup B_1^-$, with coefficients bounded by $C$ and satisfies that
	\begin{align*}
&\po_i (\bar{a}^{ij} \po_j   p) =\overline{H} + \po_i \overline{G}^i & \text{in } \ B_1,
\end{align*}	
and
	\begin{align}
&|u(\xx) - p(\xx)| \le C |\xx|^{1+\bar{\ga}}, & \xx \in B_1.
\end{align}
\end{proposition}

We refer to \cite{Liyy} for details of the proof.

\begin{lemma}\label{lemma1}
Suppose that $a^{ij}$ satisfies the uniform ellipticity condition \eqref{unifelliptic} and $a^{ij}, g^i \in C^{\ga}(\overline{B_2^+}) \cap C^{\ga}(\overline{B_2^-}), h \in L^{\infty}(B_2) $ and
	\begin{align*} 
\max_{I=+,-; i,j=1,\cdots, n}\|a^{ij}\|_{0,\ga; B_2^I} \le \Lambda,
\end{align*}
where $\Lambda$ is the same constant as in condition \eqref{unifelliptic}.
 Let $u \in H^1(B_2)$ be a solution to \eqref{ellipticu1}.
Then there exists a constant $C$, depending on $n,\ga, \gl,\Lambda$, such that
\begin{align}\label{est-inter-u}
 \max_{I=+,-}\|u\|_{1,\ga; B_1^I} & \le C (\|u\|_{0,0; B_2 } + \|h\|_{L^{\infty}(B_2) } + \max_{i=1,\cdots, n; I=+,-}\|g^i\|_{0,\ga; B_2^I }  ).
\end{align}
\end{lemma}

\begin{proof}
We will first obtain $C^{1,\ga}$ estimate for $u$ restricted on the interface ${D_{\frac{3}{2}}}$. 

	For any given  point $\xx_0 \in  D_{\frac{3}{2}}$, consider $B_{4d_0}(\xx_0)$, where  $d_0$ is sufficiently small,   to be determined later.
Set
	\begin{align} 
C_0:= \|u\|_{0,0; B_2 } + \|h\|_{L^{\infty}(B_2) } + \max_{i=1,\cdots, n; I=+,-}\|g^i\|_{0,\ga; B_2^I }. \label{def-co}
\end{align}
 We rescale domain  $B_{4d_0}(\xx_0)$ to  $B_4$   by coordinate transformation $\yy = (\xx -\xx_0)/d_0$ and set 
		\begin{align} 
		\hat{u} (\yy) = u(\xx_0 + d_0 \yy)/C_0,&&  	\hat{a}^{ij} (\yy) = a^{ij}(\xx_0 + d_0 \yy), \label{def-uhatahat}\\
		 	\hat{g}^i (\yy) = d_0 g^i(\xx_0 + d_0 \yy)/C_0, &&	\hat{h} (\yy) = d_0^2 h(\xx_0 + d_0 \yy)/C_0.\label{def-ghathhat}
	\end{align}
Then $	\hat{u}$ satisfies that
	\begin{align}\label{ellipticu0}
	&\po_i (\hat{a}^{ij}  \po_j \hat{u} ) = \hat{h} + \po_i 	\hat{g}^i  & \text{in } \ B_4.
	\end{align}
Define
	\begin{align}\label{def-aijbar}
	\bar{a}^{ij} (\xx) &=	\hat{a}^{ij} (\0 \pm) = a^{ij}(\xx_0 \pm) ,&
	\overline{H}  (\xx) &\equiv 0  ,&
	\overline{G}^i (\xx) &=	\hat{g}^i (\0 \pm),
\end{align}	
for $\xx \in B_4^{\pm}$.  
We will verify that by choosing $d_0$ sufficiently small, conditions \eqref{propcon1}--\eqref{propcon4} in Proposition \ref{prop1}  are satisfied, where  $u = \hat{u}, a^{ij} = \hat{a}^{ij}, g^i = 	\hat{g}^i, h = \hat{h}, \bar{\ga} = \ga$.  

By the definitions \eqref{def-co}--\eqref{def-ghathhat} of $C_0$, $\hat{u}$, $\hat{g}^i$ and $\hat{h}$, it is obvious that conditions 
\eqref{propcon1} and \eqref{propcon4} hold, provided that $d_0 <1$.

Then we verify condition \eqref{propcon2} as follows. Since $a^{ij}$ is $C^{\ga}$ in each half ball $B_2^{\pm}$ with $C^{\ga}$ norms bounded by $\Lambda$, we have the following estimates for $\xx \in B_4^{\pm}$:
\begin{align*} 
 |\hat{a}^{ij} (\xx) - \bar{a}^{ij} (\xx)| =   |a^{ij} (\xx_0 + d_0\xx) - a^{ij} (\xx_0 \pm)|\le \Lambda |d_0 \xx|^{\ga},
\end{align*}
leading to 
\begin{align*} 
\|\hat{a}^{ij}   - \bar{a}^{ij} \|_{Y^{1+\ga,q}(B_4)}  &= \sup_{0<r\le 1} r^{-\ga} \left(\intbar_{rB_4}  | \hat{a}^{ij}   - \bar{a}^{ij} |^q \right)^{1/q} \\
& \le \sup_{0<r\le 1} r^{-\ga} \left(\frac{1}{|B_{4r}|}\int_{B_{4r}}    \Lambda^q | d_0 \xx|^{\ga q}   \dx  \xx\right)^{1/q} \\
&\le \Lambda d_0^{\ga }.
\end{align*}
Hence, condition \eqref{propcon2} holds for sufficiently small $d_0$, depending on $\Lambda, \ga, \ve_0$. Same arguments apply to the estimates leading to condition \eqref{propcon3}. 

Noticing that the estimates above are independent of $q$, we may choose $q= 2n$. Thus, by Proposition \ref{prop1}, 
there exists a positive constant $C$, depending on $n,\ga, \gl,\Lambda$, and a continuous,  piecewise	linear function $p(\xx)$, which is linear in both $B_2^+$ and $B_2^-$ and whose coefficients are uniformly bounded by $C$, satisfying
	\begin{align*}
&\po_i (\bar{a}^{ij}  \po_j p ) =  \po_i 	\overline{G}^i & \text{in } \ B_1,
\end{align*}
with the following estimate
	\begin{align}\label{est-u-p}
&|\hat{u}(\xx) - p(\xx)| \le C|\xx|^{1+\ga},  & \forall\, \xx \in B_1.
\end{align}	
	Estimate \eqref{est-u-p} directly implies that 
		\begin{align}
		& \hat{u}(0) =   p (0 ),  &&  \label{dudp}\\	
	&D\hat{u}(0+) = D p (0+ ),  &&  D\hat{u}(0-) = D p (0- ),\label{dudp1}\\
	&|D\hat{u}(0+)| \le C, &&	|D\hat{u}(0 -)| \le C  \label{duc}.
	\end{align}
We will  obtain $C^{1,\ga}$ estimates for $u$ up to the interface $D_1$ in each subdomain $B_1^+$ and $B_1^-$. In order to achieve that, we will first consider domain $B^+_2$ and obtain the $C^{1,\ga}$ estimates for   $u(\hat{\xx}, 0+)$ on $D_{\frac{3}{2}}$, where $\hat{\xx} = (x_1, \cdots, x_{n-1})$.
Denote 
\begin{align*}
v(\xx) : =u(\xx) - u(\xx_0) - Du(\xx_0+)\cdot (\xx - \xx_0)
\end{align*}
and we know that $v$ satisfies
\begin{align}\label{ellipticv}
&\po_i (a^{ij}  \po_j v ) =  h + \po_i\bar{g}^i  , & \text{in } \ B_2^+,
\end{align}
where
\begin{align*}
\bar{g}^i=  g^i- \po_j  u(\xx_0+)a^{ij}.
\end{align*}

	For any  point $\xx_0 \in D_{\frac{3}{2}}, \xx \in B^+_{d_0}(\xx_0)$, let
	\begin{align*}
&\xx  = \xx_0 + d_0 \yy.
\end{align*}	
	Then $\yy \in B_1^+$ and  
estimates \eqref{est-u-p}--\eqref{duc} imply that the following holds
	\begin{align}\label{est-u-u0-du}
|v(\xx) |  =  C_0 | \hat{u} (\yy) - p(\yy)|  
 \le C C_0|\yy|^{1+\ga} = \frac{C C_0}{d_0^{1+\ga}}|\xx - \xx_0|^{1+\ga} .
\end{align}

Suppose another point  $\yy_0 \in D_{\frac{3}{2}}$ with 
\begin{align*}
d:= \dist(\xx_0, \yy_0) < \frac{1}{4}d_0.
\end{align*}
Set
\begin{align*}
\bar{\xx} = \xx_0 + (0,\cdots, 0,2d), \quad  Q_1= B_d(\bar{\xx}), \quad Q_2= B_{2d}(\bar{\xx}).
\end{align*}
By combining the interior H\"older estimate for solutions of Poisson's equations (see \cite{gt}, Theorem 4.15 and estimate (4.45)) and a standard perturbation argument (see the proof in \cite{gt}, Theorem 6.2), we obtain the following  Schauder interior estimate 
\begin{align}\label{est-v}
&\|v\|'_{1,\ga;Q_1} \le C \left(\|v\|_{0,0;Q_2}+ d^2\| h \|_{0,0;Q_2}+ d\max_{i=1,\cdots, n}\|\bar{g}^i- \bar{g}^i(\xx_0+)  \|'_{0,\ga;Q_2}\right),
\end{align}		
where the $\| \cdot \|'$ norm is defined as follows: Let $\DD$ be a domain and $u$ a function defined in $\DD$, $d= \mbox{diam}\, \DD$,
\begin{align*}
	\| u\|'_{k; \DD} &=\sum_{j=0}^k d^j[ u]_{j,0;\DD}\\
	\| u\|'_{k,\ga;\DD} &=\| u\|'_{k; \DD}+ d^{k+\ga}[ u]_{k,\ga;\DD}.
\end{align*}
Obviously, $Q_2 \subset B_1 (\xx_0)$ and then  estimates \eqref{duc}, \eqref{est-u-u0-du} and \eqref{est-v} imply that
\begin{align}\label{est-dvbar}
&|Dv(\bar{\xx})| = |Du(\bar{\xx}) - Du(\xx_0+)| \le CC_0 d^{\ga}.
\end{align}	
The same argument can be applied to $\yy_0$. That is, set
	\begin{align*}
\hat{v}(\xx) : =u(\xx) - u(\yy_0) - Du(\yy_0+)\cdot (\xx - \yy_0).
\end{align*}
Similar Schauder interior estimate as \eqref{est-v} gives rise to 
\begin{align}\label{est-dvhat}
&|D\hat{v}(\bar{\xx})| = |Du(\bar{\xx}) - Du(\yy_0+)| \le CC_0 d^{\ga}.
\end{align}	

Estimates \eqref{est-dvbar} and \eqref{est-dvhat} lead to 
\begin{align}\label{est-ubd}
&|Du(\xx_0+) - Du(\yy_0+)| \le CC_0 d^{\ga} = CC_0 |\xx_0- \yy_0|^{\ga}.
\end{align}
Set  $\phi (\hat{\xx}) = u(\hat{\xx}, 0+) |_{D_2}$, then estimate 	\eqref{est-ubd} implies the estimate for the boundary data $\phi$:
\begin{align}\label{est-phi}
&|\phi|_{1,\ga;D_{\frac{3}{2}} }\le  C C_0.
\end{align}
Then  we have the following boundary estimate:
\begin{align}
\|u\|_{1,\ga; B_1^+} &\le C \left(\|u\|_{0,0;B_2^+}+ \|\phi\|_{1,\ga; D_{\frac{3}{2}}}  +\| h \|_{0,0;B_2^+}+  \max_{i=1,\cdots, n}\| g^i   \|_{0,\ga;B_2^+}\right) \nonumber \\
&\le C C_0.\label{est-bdb1}
\end{align}

Applying the same argument to the lower domain $B_2^-$ and combining with estimate \eqref{est-bdb1} gives the interior estimate \eqref{est-inter-u}.
\end{proof}

\section{Boundary and global estimates}
 Let 
\begin{align*}
&W_r  = W \cap B_r ,&& W^I_r  = W^I \cap B_r , \\
&T_r   =   \po W \cap B_r  ,&&T_r^I  = \po W \cap B_r^I ,
\\
&  \EE^0=\{\xx \in \R^n:  r=0\},&&  \EE^0_r=  \EE_0 \cap B_r,
\end{align*}
where $I=+,-$ and $W, W^{\pm}$ are defined by \eqref{def-wedge} and \eqref{domain-w}.

\begin{lemma}\label{lemma2}
	Suppose that $a^{ij}$ satisfy the uniform ellipticity condition \eqref{unifelliptic}, $a^{ij}$, $g^i \in C^{\ga}(\overline{W_2^+}) \cap C^{\ga}(\overline{W_2^-})$, $h \in L^{\infty}(W_2) $, $\varphi \in C^{1,\ga}(\overline{T_2^+})\cup C^{1,\ga}(\overline{T_2^-})\cap C(\overline{T_2}) $. Let $u \in H^1(W_2)$ be a solution to 
	\begin{align}\label{ellipticuw}
	&\po_i (a^{ij}  \po_j u ) = h + \po_i g^i & \text{in } \ W_2,\\
	&u|_{T_2} = \varphi. &  
	\end{align}

	Then there exist constants $\gb \in (0,1)$ and  $C>0$, depending on $n, \gl,\Lambda, \theta_+,\theta_-$,  such that
	\begin{align}\label{est-bd-u}
&	\max_{I=+,-}\|u\|^{(-\gb;\EE_1^0)}_{1,\ga; W_1^I}\nonumber\\
 \le{}& C (\|u\|_{0,0;W_2 } + 	\max_{I=+,-}\|\varphi\|_{1,\ga; T_2^I} +  \|h\|_{L^{\infty}(W_2) } + \max_{i=1,\cdots, n; I=+,-}\|g^i\|_{0,\ga; W_2^I }  ).
	\end{align}
\end{lemma}

\begin{proof}
Denote
	\begin{align*}
C^* := \|u\|_{0,0;W_2 } + 	\max_{I=+,-}\|\varphi\|_{1,\ga; T_2^I} +  \|h\|_{L^{\infty}(W_2) } + \max_{i=1,\cdots, n; I=+,-}\|g^i\|_{0,\ga; W_2^I } .
	\end{align*}
It is easy to see that $\varphi$ is Lipschitz on $T_2$ with the following estimate	
	\begin{align}
		\|\varphi\|_{0,1; T_2} \le  C\max_{I=+,-}\|\varphi\|_{1,\ga; T_2^I},\label{est-philip}
	\end{align}
where $C$ is a constant depending on $\theta_+,\theta_-$. Hence, by the boundary  estimates of De Giogi-Nash(\cite{gt}, Theorem 8.29), there exists  $\gb \in (0,1)$, depending on $n,\ga, \gl,\Lambda, \theta_+,\theta_-$, such that $u \in  C^{\gb}(W_{\frac{3}{2}})$ and 
	\begin{align}
	\|u\|_{0,\gb; W_{\frac{3}{2}}} \le   C (\|u\|_{0,0;W_2 } + \|\varphi\|_{0,1; T_2} + \|h\|_{L^{n}(W_2) } + \max_{i=1,\cdots, n}\|g^i\|_{L^{2n}(W_2) }) 
 \le C C^*,\label{est-ubeta}
	\end{align}
where $C$ is a constant depending on 	 $n, \gl,\Lambda, \theta_+,\theta_-$.
Denote
	\begin{align*}
\varphi^*(\xx') := \varphi|_{\EE^0}  = \varphi(0,0,\xx')  .
\end{align*}
Since $\varphi$ is $C^{1,\ga}$  up to boundary $\EE^0_2$ on $\overline{T_2^+}$, it follows that $\varphi^*$ is $C^{1,\ga}$ on $\{\xx' :|\xx'|<2 \}$. Set 
	\begin{align*}
v(\xx)= u(\xx)- \varphi^*(\xx')
\end{align*}
and $v$ satisfies
	\begin{align*} 
	&\po_i (a^{ij}  \po_j v ) = h + \po_i\bar{g}^i ,& \text{in } \ W_2,\\
&v|_{T_2} = \varphi - \varphi^*, &  
\end{align*}
where $ \bar{g}^i:=  g^i - a^{ij}  \po_j \varphi^* $.
In particular $v|_{\EE} =0$ and  by estimate \eqref{est-ubeta}, we have
\begin{align} \label{est-vgb}
|v(\xx)|= |v(r,\theta, \xx')| \le 	\|u\|_{0,\gb; T_{\frac{3}{2}}} r^{\gb} \le CC^*r^\gb,
\end{align}
for any $\xx \in W_{\frac{3}{2}}$.
Then we use Schauder interior and boundary estimates, together with interior estimate \eqref{est-inter-u} across the discontinuity surface, to obtain weighted Schauder estimates up to the corner as follows.

Set 
	\begin{align*} 
\theta^* := \frac{1}{8} \min \{\theta_+, -\theta_-\}.
	\end{align*}
We divide $W_1$ into three domains to suit different types of estimates as follows. Any point $\xx_0 = (r_0, \theta_0, \xx'_0)  \in W_1 $ falls into one of the following three cases: 
\begin{flalign*}
\textit{Case 1.} & \quad -\theta^* < \theta_0 <\theta^*;&\\
\textit{Case 2.} &\quad  \theta^* \le \theta_0  \le \theta_+ - \theta^* \quad \text{or}\quad \theta_- + \theta^* \le \theta_0  \le   - \theta^* ;&\\
\textit{Case 3.}& \quad   \theta_+ - \theta^*< \theta_0 < \theta_+ \quad \text{or}\quad \theta_-  < \theta_0 <\theta_- + \theta^* . &
\end{flalign*}
Let 
	\begin{align*} 
	d:=  2 r_0\sin \theta^*.
	\end{align*}
In \textit{Case 1}, let $\hat{\xx}$ be the projection point of $\xx_0$ onto $D_1$, i. e.,
	\begin{align*} 
\hat{\xx}=   (r_0 \cos \theta_0, 0, \xx'_0)
\end{align*}
in cylindrical coordinates. Obviously, $\xx_0 \in B_d(\hat{\xx})$. By coordinate transformation
\begin{align*} 
\xx  = \hat{\xx} +d \yy,
\end{align*}
we rescale $B_d(\hat{\xx})$ into $B_1$. Set
\begin{align} 
\hat{v} (\yy) = v(\hat{\xx} + d\yy),\label{scaling1}&&  	\hat{a}^{ij} (\yy) = a^{ij}(\hat{\xx} + d\yy), \\
\hat{g}^i (\yy) = d  \bar{g}^i(\hat{\xx} + d\yy), &&	\hat{h} (\yy) = d^2 h(\hat{\xx} + d\yy).\label{scaling2}
\end{align}
Then $\hat{v}$ satisfies 
	\begin{align*}
&\po_i (\hat{a}^{ij}  \po_j \hat{v} ) = \hat{h} + \po_i 	\hat{g}^i 
\end{align*}
in $B_2$. Hence, estimate \eqref{est-inter-u} in Lemma \ref{lemma1} leads to
\begin{align*}
\max_{I=+,-}\|\hat{v}\|_{1,\ga; B_1^I}
 \le C (\|\hat{v}\|_{0,0; B_2 } + \| \hat{h}\|_{L^{\infty}(B_2) } + \max_{i=1,\cdots, n; I=+,-}\|\hat{g}^i\|_{0,\ga; B_2^I }  ).
\end{align*}
Scaling domain $B_1$ back to  $B_d(\hat{\xx})$ and setting 
 \begin{align*}
Q_r = B _r(\hat{\xx}),\quad Q_r^I= B^I_r(\hat{\xx}) , \quad I=+,-, r>0,
 \end{align*}
 we obtain
\begin{align*}
 \max_{I=+,-}\|v\|'_{1,\ga;Q_1^I} \le C \left(\|v\|_{0,0;Q_2}+ d^2\| h \|_{0,0;Q_2}+ d\max_{i=1,\cdots, n,I=+,-}\|\bar{g}^i  \|'_{0,\ga;Q_2^I}\right).
\end{align*}
Therefore, by estimate \eqref{est-vgb} and the definition of $\bar{g}^i$, we have
\begin{align}\label{est-vgbinter}
\max_{I=+,-}\|v\|'_{1,\ga;Q_1^I} \le C C^* d^\gb.
\end{align}

In  {\it Case 2}, set $Q_r^I = B _r(\xx_0)$, where $I=\pm$, if $ B _r(\xx_0 ) \subset W^{\pm}$, and use the same arguments as for estimate \eqref{est-v}, we obtain the following Schauder interior estimate:
\begin{align}\label{est-vint}
\|v\|'_{1,\ga;Q_1^I} \le C \left(\|v\|_{0,0;Q_2^I}+ d^2\| h \|'_{0,0;Q_2^I}+ d\max_{i=1,\cdots, n}\|\bar{g}^i \|'_{0,\ga;Q_2^I}\right)\le C C^* d^\gb.
\end{align}		

In  {\it Case 3}, we project $\xx_0$ onto the boundary $\po W$ and denote the projection point by $\hat{\xx}$. Set 
\begin{align*}
Q_r^I = B _r(\hat{\xx}) \cap W^I, \quad \hat{T}_r^I =  B _r(\hat{\xx}) \cap \po W^I, \quad I=+,-.
\end{align*}
Schauder boundary estimates gives rise to
	\begin{align}
	\|v\|'_{1,\ga;Q_1^I } & \le C (\|v\|_{0,0;Q_2^I } + 	\|\varphi\|'_{1,\ga; \hat{T}^I_2} +  d^2\|h\|_{L^{\infty}(Q^I_2) } + d\max_{i=1,\cdots, n; }\|g^i\|_{0,\ga; Q_2^I }  ) \nonumber\\
	& \le C C^* d^\gb.\label{est-bd-v}
	\end{align}

We will use  estimates \eqref{est-vgbinter} --  \eqref{est-bd-v} above to obtain the corner estimate  \eqref{est-bd-u}. Noticing the definitions of $\theta^*$ and $d$ above, we have
\begin{align*}
    d    \le r_0 = \gd_{\xx_0}  \le Cd.
\end{align*}
We first estimate  $	{[ v ]}_{1,\ga;W_1^+}^{(-\gb)}$ as follows. For any $\xx_0= (r_0,\theta_0, \xx_0'), \xx=(r_1,\theta_1, \xx') $	in $W_1^+$, assume $r_0 \le r_1$ and we have
\begin{align*}
 d \le 	\gd_{\xx_0,\xx}= r_0  \le Cd.
\end{align*}  
Thus,  estimates \eqref{est-vgbinter} --  \eqref{est-bd-v} imply that	
	\begin{align*}	
	{[ v  ]}_{1,\ga;W_1^+}^{(-\gb)}= {}&
	\sup_{
		\begin{subarray}{c}
			\xx, \xx_0\in W_1^+\\
			\xx \ne \xx_0
	\end{subarray}}
	\left((\gd_{\xx,\xx_0})^{\max(1+\ga-
		\gb,0)} \frac{|Dv(\xx)-D
	v(\xx_0)|}{|\xx-\xx_0|^\ga}\right)\\
= {}&
\sup_{
	\begin{subarray}{c}
	\xx, \xx_0\in W_1^+\\
	0<|\xx - \xx_0|<d
	\end{subarray}}
\left(r_0^{1+\ga-
	\gb} \frac{|Dv(\xx)-D
v(\xx_0)|}{|\xx-\xx_0|^\ga}\right)	\\
&+\sup_{
	\begin{subarray}{c}
	\xx, \xx_0\in W_1^+\\
|\xx - \xx_0|\ge d
	\end{subarray}}
\left(r_0^{1+\ga-
	\gb} \frac{|Dv(\xx)-D
	v(\xx_0)|}{|\xx-\xx_0|^\ga}\right)\\
\le{}& C \sup_{ \xx_0\in W_1^+}\left( d^{-\gb} 	\|v\|'_{1,\ga;Q_1^+} + d^{1-\gb} \|D v\|_{L^{\infty}(Q_1^+)}\right)\\
\le{}& CC^*.
		\end{align*}
Similarly, we have the estimate in $W_1^-$:
	\begin{align*}	
	{[ v  ]}_{1,\ga;W_1^-}^{(-\gb)} \le CC^*.
	\end{align*}		
It is easy to obtain the estimates for $	{[ v  ]}_{0,0;W_1^I}^{(-\gb)} $	and  $	{[ v  ]}_{1,0;W_1^I}^{(-\gb)} $:
		\begin{align*}	
		{[ v  ]}_{0,0;W_1^I}^{(-\gb)}+	{[ v  ]}_{1,0;W_1^I}^{(-\gb)}  \le CC^*.
		\end{align*}	
These lead to the corner estimate \eqref{est-bd-u}.
\end{proof}

We now use the interior (Lemma \ref{lemma1}) and corner estimates  (Lemma \ref{lemma2}) to obtain the global estimate in Theorem \ref{thmmain}.

\begin{proof}[Proof of Theorem \ref{thmmain}]
Denote
\begin{align*}
\hat{C}:=  	\max_{I=+,-}\|\varphi\|_{1,\ga;   \GO ^I}  
 +  \|h\|_{L^{\infty}( \GO) } + \max_{i=1,\cdots, n; I=+,-}\|g^i\|_{0,\ga; \GO^I }  .
\end{align*}
By the wedge  condition \eqref{def-wedge} at $\EE$, there exists a constant $r^*>0$, such that for any $\xx_0 \in \EE$, $\exists \chi \in C^{1,\ga}(B_{r^*}(\xx_0))$, which is an isomorphism from $B_{r^*}(\xx_0)$ to $\chi(B_{r^*}(\xx_0))$ satisfying $W_2 \subset \chi(B_{r^*}(\xx_0) \cap \GO)$. Let $r_*>0$ be the radius such that $B_{r_*}(\xx_0)\cap \GO  \subset \chi^{-1} (W_1)$. 

For any $\xx \in \GO$, if $\dist (\xx, \EE) < r_*$, we can find some point $\xx_0 \in \EE$ such that $|\xx - \xx_0| = \dist (\xx, \EE)$. Then $\chi (B_{r_*}(\xx_0)\cap\GO ) \subset W_1$ and corner estimate \eqref{est-bd-u}  gives rise to 
\begin{align*}
 	\max_{I=+,-}\|u\|^{(-\gb;\EE)}_{1,\ga;B_{r_*}(\xx_0)\cap\GO^I} 
\le{}&  C (\|u\|_{0,0;B_{r^*}(\xx_0) } + 	\max_{I=+,-}\|\varphi\|_{1,\ga; B_{r^*}(\xx_0)\cap(\po \GO)^I} \nonumber\\
&+  \|h\|_{L^{\infty}(B_{r^*}(\xx_0)) } + \max_{i=1,\cdots, n; I=+,-}\|g^i\|_{0,\ga; B_{r^*}(\xx_0)\cap(\po \GO)^I }  ) \nonumber\\
\le{}&  C (\|u\|_{0,0;\GO } +\hat{C}).
\end{align*}
By the weak maximum principle({\it cf.} \cite{gt}, Theorem 8.16), we know that
\begin{align*}
\|u\|_{0,0;\GO } \le C\hat{C}.
\end{align*}
Therefore, we conclude that
\begin{align}
\label{est-bd-ugo}
	\max_{I=+,-}\|u\|^{(-\gb;\EE)}_{1,\ga;B_{r_*}(\xx_0)\cap\GO^I} 
	\le C\hat{C}.
\end{align}
If $\dist (\xx, \EE) \ge  r_*$ and $\dist (\xx, \GC) < r_{**}:=\frac{1}{4}r_*\sin \theta_* $, we can find  $\xx_0 \in \GC$ such that $\|\xx - \xx_0\| = \dist (\xx, \GC)$. we apply the interior estimate  \eqref{est-inter-u} to obtain 
\begin{align}
\label{est-int-ugo}
\max_{I=+,-}\|u\|_{1,\ga;B_{r_{**}}(\xx_0)\cap\GO^I} 
\le C\hat{C}.
\end{align}
For the rest of $\xx$, the classical Schauder interior and boundary estimates apply and we have
\begin{align}
\label{est-schauder-ugo}
\|u\|_{1,\ga;B_{r_{**}/2}(\xx)\cap\GO} 
\le C\hat{C}.
\end{align}
Estimates \eqref{est-bd-ugo},\eqref{est-int-ugo} and \eqref{est-schauder-ugo} combined together give rise to the global estimate \eqref{est-glo-u}, thus complete the proof of Theorem \ref{thmmain}.
\end{proof}

\section{An example to illustrate the corner issue}

We obtained the $C^\gb$ estimate up to the edge in the proof of Theorem \ref{thmmain}. For the purpose of application to nonlinear problems, sometimes we do need $C^{1,\ga}$ estimates up to the corners. In this section, we will provide a simple example to show that even the coefficients are piecewise constant, the regularity at the corner is a complicated issue  depending on the boundary shape and the jump of the coefficients.

In general, the regularity of the solution up to the corner depends on both the boundary shape near the corner and the coefficients of the elliptic equations. We do not elaborate the general situation and only focus on equation with a simple jump on coefficients. We will investigate   how the angles between the boundary and discontinuity surface affect the regularity of the solutions near the corners. 

Let 
\begin{align}
\GO = W \cap B_1,&& \GO^+ = W \cap B^+_1, &&\GO^- = W \cap B^-_1, \label{domainex}
\end{align}
where $W$ is defined in \eqref{def-wedge} and  the angles between the boundary and discontinuity line are $\theta_+, \theta_-$, which will  affect the regularity of solutions at the corner $\0$. 

Consider the following elliptic equation
	\begin{align}\label{ellipticex}
&\div (a(\xx)  \nabla u ) = 0  & \text{in } \ \GO,
\end{align}
where
\begin{align*}
a(\xx) = \left\{
\begin{array}{ll}
a_0, & \xx \in \GO^+ \\
1, & \xx \in \GO^- 
\end{array}
\right.
\end{align*}
and $a_0 $ is a positive constant to be determined later.

\begin{remark}
The equations with this type of coefficients were studied in \cite{BonVo}. Caffarelli and his collaborators studied interface transmission problems in \cite{CaffSS}, which has essentially the same type of  coefficients as described above.  In \cite{Bae,cxz,schen1,schen2}, the background states are also solutions to the elliptic equations with piecewise constant coefficients. 
\end{remark}

We construct solution $u$ to \eqref{ellipticex} in the following form:
\begin{align} \label{def-exu}
	u(\xx) = \left\{
	\begin{array}{ll}
		u^+(\xx)= r^\gc(A
		\sin \gc \theta + B\cos \gc \theta), & \xx \in \GO^+ \\
		u^-(\xx)= r^\gc (C
	\sin \gc \theta + 	D\cos \gc \theta),  & \xx \in \GO^- 
	\end{array}
	\right.,
\end{align}
where   $\gc, A,B,C,D$ are constants and $\gc >0$.
Obviously, $u$ satisfies 
\begin{align*}
\GD u =0
\end{align*}
in each subdomain $\GO^+$ or $\GO^-$. To be a weak solution to \eqref{ellipticex} in the whole domain $\GO$, two conditions should be satisfied on the discontinuity line $\{(r,\theta): r>0, \theta =0\}$: one is the continuity of $u$ across the discontinuity line; the other is the jump condition $D u^+ \cdot \nn = D u^- \cdot \nn$ on the discontinuity line, where $\nn$ is the normal direction on the line. Thus we have
\begin{align} 
	u^+|_{\theta=0}& =  	u^-|_{\theta=0}, \label{discon1}\\
   \po_\theta{u^+}|_{\theta=0} &=   \po_\theta	u^-|_{\theta=0}.\label{discon2}
\end{align}
Conditions \eqref{discon1} and \eqref{discon2} imply that
\begin{align*}
&B=D,  \qquad a_0 A =C.
\end{align*}
Without loss of generality, we may take 
\begin{align*}
	&B=D =1.
\end{align*}
We want $u=0$ on the boundary $\{\theta= \theta_+\}$ and  $\{\theta= \theta_-\}$. So we solve
\begin{align*}
\left\{ 	\begin{array}{ll}
 u^+ (r, \theta_+)=0\\
	u^-(r, \theta_-)=0 
\end{array}
\right.
\end{align*}
for $A, a_0$ and obtain
\begin{align*}
A = -\frac{\cos \gc \theta_+}{\sin \gc \theta_+}, &&   a_0 = \frac{\sin \gc \theta_+\cos \gc \theta_-}{\cos \gc \theta_+\sin \gc \theta_-}.
\end{align*}
Now we take
	\begin{align*}
& \gc =   \frac{4}{5},\qquad \theta_+ = \frac{3}{4} \pi, \qquad  \theta_- = -\frac{1}{4} \pi .
\end{align*}
Then the wedge wall becomes a straight segment, on which $u=0$. This shows that even both the boundary  and the boundary data are smooth, we still have non-smooth solution due to the discontinuity of the coefficient $a(\xx)$. In this example the solution is $C^{\frac{4}{5}}$ up to the corner $\0$.

\begin{proposition}\label{propex}
	Suppose that domain $\GO$ is defined by \eqref{domainex}, $\theta_+ \in (0,\frac{\pi}{2}), \theta_- \in (-\frac{\pi}{2}, 0)$ and $u|_{\po W \cap B^I_{1/2}} \in C^{1,\ga}(\overline{\po W \cap B^I_{1/2}})$, where $I=+,-$ and  $\ga \in (0,1)$ depending on $\theta_+,\theta_-,  a_0$.   Then   the solution $u \in C^{1,\ga} (\overline{ B^+_{1/2}}) \cap C^{1,\ga}  (\overline{ B^-_{1/2}}) $.
\end{proposition}

\begin{proof}
Denote the tangential directions at $\0$ along the upper and lower boundaries as
	\begin{align*}
& \gtt^+ =  (\cos \theta_+ , \sin   \theta_+ ), &&  \gtt^- =  (\cos \theta_- , \sin   \theta_- ).
\end{align*}
 We will find a piecewise linear function $p$, as a solution to \eqref{ellipticex},  such that the tangential derivatives of $p$ and $u$ are equal at $\0$. Set
 	\begin{align*}
 & c^+ =  D_{\gtt^+} u(0,0), &&  c^- =  D_{\gtt^-} u(0,0),
 \end{align*}
and 
\begin{align*} 
p(\xx) = \left\{
\begin{array}{ll}
p^+(\xx)= a^* x_1 +b^+x_2, & \xx \in \GO^+ \\
p^-(\xx)= a^* x_1 +b^- x_2,  & \xx \in \GO^- 
\end{array}
\right. .
\end{align*}
Then we have 
\begin{align} \label{eqn-astarb1}
& D_{\gtt^+} p(\xx) = \cos \theta_+ a^* + \sin   \theta_+ b^+ =c^+,\\
& D_{\gtt^-} p(\xx) = \cos \theta_- a^* + \sin   \theta_- b^- =c^-.
\end{align}
Since $p$ is a solution to \eqref{ellipticex}, the jump condition on the discontinuity line should be satisfied:
\begin{align} \label{eqn-astarb2}
&  a_0 \po_{x_2 }p^+ (\xx)  -   \po_{x_2 }p^- (\xx) = a_0 b^+-  b^- =0.
\end{align}
It is obvious that the linear system \eqref{eqn-astarb1}--\eqref{eqn-astarb2} is uniquely solvable for $a^*, b^+,b^-$ if and only if 
	\begin{align*}
  \cos \theta_+ \sin   \theta_- a_0 - \sin   \theta_+ \cos \theta_- \neq 0.
\end{align*}
Thus the assumptions  $a_0 >0, \theta_+ \in (0,\frac{\pi}{2}), \theta_- \in (-\frac{\pi}{2}, 0)$ guarantee the solvability of $a^*, b^+,b^-$.

Let 
\begin{align} \label{def-v}
v(\xx) = u(\xx) - u(0,0) - p(\xx).
\end{align}	
Then $v$ is also a solution to \eqref{ellipticex} and 
	\begin{align*}
\left|v(\xx)|_{\po W \cap B_1}\right| \le C r^{1+\ga}.
\end{align*}
	
Define a barrier function $w$ by
\begin{align} \label{def-w}
w(\xx) = C r^{1+\ga} \cos ((1+\ga +\tau_0) \theta), 
\end{align}	
where 
\begin{align*}  
 0<\ga ,\tau_0 <1, \quad 1+\ga +\tau_0 < \min\left\{\frac{\pi}{2\theta_+},-\frac{\pi}{2\theta_-}\right\}.
\end{align*}		
By the comparison principle, we conclude that
\begin{align}  
|v(\xx)| \le Cr^{1+\ga}. \label{est-v1ga}
\end{align}
Use similar scaling as in the proof of Lemma \ref{lemma2}, estimate \eqref{est-v1ga} and  interior estimate \eqref{est-inter-u}  lead to  $C^{1,\ga}$ regularity up to corner $\0$.

\end{proof}

\begin{remark}
In the proof of Proposition \ref{propex}, the construction of the barrier function $w$ is crucial to the $C^{1,\ga}$ regularity at $\0$. We can choose proper $\ga$ to obtain the barrier function   $w$ due to the assumption     $ \theta_+ \in (0,\frac{\pi}{2}), \theta_- \in (-\frac{\pi}{2}, 0)$. For other coefficients, the angle ranges may vary in order to obtain $C^{1,\ga}$ regularity. Therefore, it becomes a case by case investigation and it is difficult to provide a universal criterion to guarantee $C^{1,\ga}$ regularity at corners.
\end{remark}

%
%


\begin{thebibliography}{99}

\bibitem{Bae}	
 M. Bae, Stability of contact discontinuity for steady Euler system in infinite duct,  {\it Z. Angew. Math. Phys.},  {\bf 64} (2013), 917--936.	



\bibitem{BonVo}
E. Bonnetier and  M. Vogelius, An elliptic regularity result for a composite medium
with ``touching'' fibers of circular cross-section, \textit{SIAM J. Math. Anal.}, \textbf{31} (2000),  651--677.


\bibitem{CaffSS}
 L.   Caffarelli, M. Soria-Carro and P. Stinga, Regularity for $C^{1,\alpha}$   interface transmission problems, \emph{Arch. Ration. Mech. Anal.},  {\bf 240} (2021), no. 1, 265--294. 

	
	
\bibitem{cxz}
J. Chen, Z.  Xin and A. Zang, 
Subsonic flows past a profile with a vortex line at the trailing edge, \textit{SIAM J. Math. Anal.}, \textbf{54} (2022), no. 1, 912--939.

\bibitem{schen1} S.-X. Chen,
Stability of a Mach configuration,
\emph{Comm. Pure Appl. Math.}, {\bf 59} (2006), 1--35.

\bibitem{schen2}S.-X. Chen, Mach configuration in pseudo-stationary compressible flow, \emph{J. Amer. Math. Soc.}, \textbf{21} (2008), no. 1, 63--100.

\bibitem{gt}
 D. Gilbarg and N. Trudinger,
\textit{Elliptic Partial Differential Equations of Second Order},  2nd
Ed., Springer-Verlag: Berlin, 1983.


\bibitem{Liyy}
  Y.-Y. Li and  M. Vogelius, Gradient estimates for solutions to divergence form elliptic equations with discontinuous coefficients, \textit{Arch. Ration. Mech. Anal.}, \textbf{153} (2000), 91--151.



\end{thebibliography}
\end{document}